\documentclass[a4paper, reqno, 12pt]{amsart}
\usepackage[english]{babel}
\usepackage[T1]{fontenc}
\usepackage{mathptmx,amsmath,dsfont}
\usepackage{hyperref}
\usepackage{graphicx}
\usepackage{amssymb,verbatim,cases}
\usepackage{fullpage}
\usepackage[svgnames,hyperref]{xcolor}
\newtheorem{theorem}{Theorem}[section]
\newtheorem{corollary}[theorem]{Corollary}
\newtheorem{lemma}[theorem]{Lemma}
\newtheorem{definition}[theorem]{Definition}
\newtheorem{proposition}[theorem]{Proposition}
\theoremstyle{remark}

\newtheorem{remark}[theorem]{Remark}

\numberwithin{equation}{section}

\newcommand{\R}{\mathbb{R}}
\newcommand{\N}{\mathbb{N}}

\newcommand{\E}{\mathcal{E}}
\newcommand{\Em}{\mathcal{E}_m}
\newcommand{\Eom}{\mathcal{E}_{0,m}}
\newcommand{\Epm}{\mathcal{E}_{p,m}}
\newcommand{\Fpm}{\mathcal{F}_{p,m}}
\newcommand{\Ecm}{\mathcal{E}_{\chi, m}}
\newcommand{\Eam}{\mathcal{E}^a_m}
\newcommand{\F}{\mathcal{F}}
\newcommand{\Fm}{\mathcal{F}_m}
\newcommand{\Fam}{\mathcal{F}^a_m}
\newcommand{\Nm}{\mathcal{N}_m}
\newcommand{\SHm}{SH_m}

\newcommand{\bC}{\mathbb{C}}

\newcommand{\ca}{cap_{m,\Omega}}

\begin{document}
\author{Do Thai Duong \textit{$^{1}$} and Van Thien Nguyen \textit{$^{2}$} }
\address{\textit{$^{1}$}~Institute of Mathematics, Vietnam Academy of Science and Technology, 18 Hoang Quoc Viet, Cau Giay, Hanoi, Vietnam}
\email{dtduong@math.ac.vn}
\address{\textit{$^{2}$}~FPT University, Education zone, Hoa Lac high tech park, Km29 Thang Long highway,
Thach That ward, Hanoi, Vietnam}
\email{thiennv15@fe.edu.vn}

\title{On the weighted $m-$energy classes}

\subjclass[2000]{Primary 32U15}
\keywords{}
\date{\today}
\maketitle
\begin{abstract}
In this article, we investigate the weighted $m-$subharmonic functions. We shall give some properties of this class and consider its relation to the $m-$Cegrell classes. We also prove an integration theorem and an almost everywhere convergence theorem for this class.
\end{abstract} 
\tableofcontents

\section{Introduction}
In 1985, Caffarelli, Nirenberg and Spruck \cite{cns} generalized the notion of subharmonic and plurisubharmonic functions  which called be $m$-subharmonic functions. A smooth function $u$ defined in open subset $\Omega$ of $\mathbb{C}^n$ is $m$-subharmonic if 
$$\sigma_k(u)=\sum_{1\leq j_1<\cdots<j_k\leq n}\lambda_{j_1}\cdots \lambda_{j_k}\geq 0, \forall k =1,\cdots,m$$
where $(\lambda_{j_1},\cdots,\lambda_{j_n})$ is the eigenvalue vector of the complex Hessian matrix of $u$.

Generally, the complex Hessian operator is neither linearly nor holomorphically invariant. This makes it is harder to study when compared to the classical Laplacian and Monge-Amp \`ere operators. The pluripotential theory of $m$-subharmonic functions developed by B\l ocki in \cite{bl05} where
the complex Hessian operator $H_m(u)= (dd^cu)^m\wedge (dd^c|z|^2)^{n-m}$ is well-defined for $m$-subharmonic function $u$ is a Radon measure (see \cite{bl05}). 

For $1\leq m \leq n,$, we will always assume that $\Omega$ is a bounded $m$-hyperconvex domain
to ensure the existence of $m$-subharmonic functions defined on $\Omega$. Since the complex Hessian operator can not be defined for all $m$-subharmonic functions, a question that arose is find a natural domain of the complex Hessian operator. Lu extended the results of Cegrell (\cite{C98, C04}) to introduce the Cegrell classes $\E_m(\Omega), \F_m(\Omega), \E_{p,m}(\Omega)$ for $m$-subharmonic functions. He showed that the
complex $m$-Hessian operator is well-defined in these classes. Moreover, he also proved
that the class $\E_m(\Omega)$ is the biggest class where the Hessian operator can be defined (see \cite{luthesis,lu13a,lu13b}).

To study the range of the complex Monge-Amp\`ere operator in the K\"ahler manifold setting, Guedj and Zeriahi in \cite{GZ07} gave a notion of weighted energy class $\E_{\chi}(X,\omega)$. This class is a subset of all $\omega$-plurisubharmonic functions with full mass  and it generalizes the class of $\omega$-plurisubharmonic functions with finite energies. Since then, there are many works investigated the weighted energy classes, e.g. \cite{Ben09, Ben11, BGZ, rafal10, HH11}, \dots

Recently, Cegrell classes for $m$-subharmonic functions have been invesgated by many authors. For more details we refer the reader to \cite{AC20, AC23, AC22, dk14, dk17, cu13, cu14, DD22, th16, NVT18, NVT19, sa, VN17}, \dots. The weighted $m$-subharmonic functions were introduced by Vu in \cite{VVH16}.

In this paper, by replying ideas from \cite{Ben09, BGZ, HH11} we shall investigate more properties of the weighted $m$-subharmonic class $\E_{m,\chi}(\Omega)$ for any increasing weight function $\chi\colon \mathbb{R}^{-} \to \mathbb{R}^{-}$. In Section \ref{sec3}, we give some inclusions between the weighted $m$-energy class and well known $m$-Cegrell classes. More details, we show that as $\chi$ is not the zero function then $\E_{\chi,m}(\Omega)\subset \E_m(\Omega)$, and $\E_{m,\chi}(\Omega)$ is a subclass of the class $\mathcal{N}_m(\Omega)$ when $\chi(-t)<0$ for all $t>0$ (see Theorem \ref{Main theorem 1}). In the case $\chi(-\infty)=-\infty$, the weighted $m$-energy functions have complex Hessian measures vanish on $m$-polar sets (see Theorem \ref{chi=-infty}). In addition if $\chi(-\infty)=-\infty$ and $\chi(0)=0$, we obtain the inclusion $\E_{m,\chi}\subset \F_m^a(\Omega)$, the Cegrell class of $m$-subharmonic functions with finite energy whose complex Hessian measures vanish on $m$-polar sets. We also give characterizations of this class in term of $\chi$-energy as well as $m$-capacity (see Theorem \ref{thm:in-ca} and Proposition \ref{pro:finiteint}). In Section \ref{sec4}, with an addition condition on $\chi$, we shall prove the convergence for the weighted $m$-energy class (see Theorem \ref{thm:convergence}). And in the rest of Section \ref{sec4}, we formulate an integration theorem that help us creating more elements in this class (see Theorem \ref{thm:integration}).

\section{Cegrell classes for $m$-subharmonic functions}
This section is a brief introduction which summarizes the definition of Cegrell classes for $m$-subharmonic functions and theirs properties. More details and proofs can be found in \cite{luthesis, th16, NVT19}.

We shall use the canonical $(1,1)$-form on $\bC^n$:
$$\omega =dd^c|z|^2=2i\sum_{j=1}^n dz_j\wedge d\bar{z}_j.$$
Denote by $\SHm(\Omega)$ the collection of all $m$-subharmonic functions on $\Omega$.
A set $E\subset\mathbb{C}^n$ is called \textit{$m$-polar} if  $E\subset \{v=-\infty\}$ for some $v\in \SHm(\bC^n)$ and $v\not\equiv -\infty$.

In pluripotential theory, capacities play an important role. We define analogous capacities in the setting of $m$-sh functions.
\begin{definition}\label{def:mcap}
Let $E$ be a Borel subset of $\Omega$. The $m$-capacity $\ca (E)$ of $E$ with respect to $\Omega$ is defined by
$$\ca(E)=\sup\left\{\int_E H_m(u), u\in SH_m(\Omega), -1\leq u\leq 0\right\}.$$
\end{definition}

Recall that a bounded domain $\Omega\subseteq \mathbb{C}^n$, is called \textit{$m$-hyperconvex} if there exists a bounded $m$-subharmonic function $\rho : \Omega \to (-\infty, 0)$ such that the closure of the set $\{z\in \Omega \colon \rho(z)<c\}$ is compact in $\Omega$, for every $c\in (-\infty, 0)$. The function $\rho$ can be scaled by a constant of proportion such that $\rho(z)\in [-1,0)$, for all $z\in \Omega$. Such $\rho$ is called \textit{the exhaustion function} of $\Omega$.

Throughout this chapter, we always assume that $\Omega$ is an $m$-hyperconvex domain. The Cegrell classes for $m$-subharmonic were introduced by Lu for classes $\Eom, \Epm, \Fm, \Fpm, \Em$ ( see\cite{luthesis}) and by the second author of this article for the class $\Nm$ ( see \cite{NVT19}).
\begin{definition}[Cegrell classes]
\normalfont
\label{def: cegrell classes}

\begin{itemize}
\item We let $\Eom(\Omega)$ to be the class of all bounded functions in $SH_m(\Omega)$
such that
$$\lim_{z\to\partial\Omega}u(z)=0 \ \mbox{and} \  \int_\Omega H_m(u) < +\infty.$$
\item For each $p > 0$, let $\Epm(\Omega)$ denote the class of all functions $u\in SH_m^-(\Omega)$ such that there exists a decreasing sequence $\{u_j\}\subset \Eom(\Omega)$ satisfying $u_j\downarrow u$ on $\Omega$ and
$$\sup_j\int_\Omega (-u_j)^pH_m(u_j)<+\infty.$$
If we require moreover that $\sup_j \int_\Omega H_m(u_j ) < +\infty$ then, by definition, $u$ belongs to $\Fpm(\Omega)$.
\item A function $u\in \SHm(\Omega)$ belongs to $\Em(\Omega)$ if for each $z_0\in \Omega$, there
exist an open neighborhood $U\subset \Omega$ of $z_0$ and a decreasing sequence $\{u_j\}\subset \Eom(\Omega)$
such that $u_j\downarrow u $ in $U$ and $\sup_j\int_\Omega H_m(u_j)< +\infty$.
\item Denote by $\Fm(\Omega)$ the class of functions $u\in SH_m(\Omega)$ such that there exists a
sequence $\{u_j\}\subset \Eom(\Omega)$ decreasing to $u$ in $\Omega$ and $\sup_j \int_\Omega H_m(u_j) <+\infty$.
\item
Let $u\in SH_m(\Omega)$, and let $\{\Omega_j\}$ be a fundamental sequence of $\Omega$, i.e., $\Omega_j\Subset\Omega_{j+1}$ and $\cup_j\Omega_j=\Omega$.
Set
$$u^j(z) = \left(\sup\{\varphi(z) \colon \varphi\in SH_m(\Omega), \varphi\leq u \ \textrm{on} \ \Omega_j^c\}\right)^*,$$
where $\Omega_j^c$ denotes the complement of $\Omega_j$ in $\Omega$.
The function $\tilde{u}$ defined by
$$\tilde{u}=\left(\lim_{j\to\infty}u^j\right)^*$$
We define $\Nm(\Omega)$ be the set of all functions $u\in\Em(\Omega)$ such that $\tilde{u}=0$.
\end{itemize}
\end{definition}

\begin{remark}\label{re:cegrell}
\begin{itemize}
\item[(i)] We have the following strictly inclusions
\begin{align*}
& \Eom(\Omega)\subset \Fpm(\Omega) \subset \Fm(\Omega)\subset \Nm(\Omega) \subset \Em(\Omega),\\
& \Eom(\Omega)\subset \Fpm(\Omega) \subset \Epm(\Omega)\subset \Em(\Omega),\\
& \SHm^{-}(\Omega)\cap L^\infty_{\text{loc}}(\Omega)\subset \Em(\Omega).
\end{align*}
There is no inclusion between $\Fm$ and $\Epm$. An example has been showed in \cite{ACH07}.

\item[(ii)] Every function $u\in \Nm(\Omega)$ has zero boundary values in sense that 
$$\limsup_{z\to \partial\Omega} u(z) = 0.$$

\item[(iii)] The class $\Eom(\Omega)$ are sometimes called test functions since each smooth function with compact support in $\Omega$ can be written as the difference of two smooth functions in $\Eom(\Omega)$.
\item[(iv)]  $\Em(\Omega)$ is the largest set of non-positive $m$-subharmonic
functions where the complex Hessian operator is well-defined.
\item[(v)] We can think $\Nm(\Omega)$ is the  class for which the
smallest maximal $m$-subharmonic majorant is identically equal to $0$.
\end{itemize}
\end{remark}

Let $\mathcal{K}$ be one of the classes $\Eom(\Omega),\Fm(\Omega),\Nm(\Omega), \Em(\Omega), \Epm(\Omega), \Fpm(\Omega)$. Denote by $\mathcal{K}^a$ the set of all functions in $\mathcal{K}$ whose Hessian measures vanish on all $m$-polar sets of $\Omega$. The following results was shown in \cite{luthesis, NVT19}.

\begin{lemma}\label{lem:convexcone}
\begin{itemize}
\item[(i)] $\mathcal{K}$ is a convex cone, i.e., if $u,v\in \mathcal{K}$ then $au + bv \in \mathcal{K}$ for arbitrary nonnegative constants $a, b$. Moreover, if $u\in \mathcal{K}, v\in \SHm^{-}(\Omega)$ then $\max(u, v)\in \mathcal{K}$.
\item[(ii)] The Hessian measures of functions in $\Epm (\Omega), \Fm (\Omega)$ vanish on $m$-polar subsets of $\Omega$. This means that $\Epm (\Omega), \Fm (\Omega)$ are proper subsets of $\Em^a(\Omega)$.
\end{itemize}
\end{lemma}

\begin{theorem}\cite[Theorem 3.1]{NVT18}\label{Inequalities for Capacity}
	Let $u\in\Fm(\Omega)$. Then for all $s,t>0$, we have
	$$t^m\text{Cap}_{m,\Omega}(\{u<-s-t\})\leq\int\limits_{\{u<-s\}}(dd^cu)^m\wedge\omega^{n-m}\leq s^m\text{Cap}_{m,\Omega}(\{u<-s\}).$$
\end{theorem}
The following theorem is a generalization of \cite[Theorem 2.1]{BGZ}. The proof is a slight modification of the plurisubharmonic case.
\begin{theorem}\label{vanish on polar sets}
	Let $u\in\Em(\Omega)$. Then for every borel set $B\subset\Omega\setminus\{u=-\infty\}$, we have
	$$\int\limits_B(dd^cu)^m\wedge\omega^{n-m}=\lim\limits_{k\rightarrow\infty}\int\limits_{B\cap\{u>-k\}}(dd^cu_k)^m\wedge\omega^{n-m},$$
	where $u_k=\max(u,-k)$. The measure $(dd^cu)^m\wedge\omega^{n-m}$ puts no mass on $m-$polar sets $E\subset\{u>-\infty\}.$
\end{theorem}
The following theorem, which is a consequence of \cite[Lemma 5.5]{NVT19}, will be used in sequel.
\begin{theorem}\label{u=v}
	Let $u,v\in\mathcal{N}_m(\Omega)$ is such that $u\geq v$ and $(dd^cu)^m\wedge\omega^{n-m}=(dd^cv)^m\wedge\omega^{n-m}$. Suppose that there exists $w\in\Em(\Omega)$ such that $w\not\equiv-\infty$ and $\int\limits_\Omega(-w)(dd^cu)^m\wedge\omega^{n-m}<+\infty$. Then $u=v$ on $\Omega$.
\end{theorem}

\section{Weighted $m-$energy classes}\label{sec3}
Let us define the weighted $m$-energy classes. In this section, we always assume the weight $\chi \colon \mathbb{R}^{-}\to \mathbb{R}^-$ is increasing.

\begin{definition}\label{def:chi}
The weighted $m$-energy class with respect to the weight $\chi$ can be defined as follows
\begin{align*}
\Ecm(\Omega)=\{u\in\SHm(\Omega):\ \exists(u_j)\in\Eom(\Omega),\ u_j\searrow u,\ \sup_j\int\limits_\Omega(-\chi)\circ u_j(dd^cu_j)^m\wedge\omega^{n-m}<\infty\}.
\end{align*}
\end{definition}
\begin{remark}
The weighted $m$-energy class generalizes the energy Cegrell classes $\Epm, \Fm, \Fpm$.
\begin{itemize}
\item[(i)] When $\chi \equiv -1$, then $\Ecm(\Omega)$ is the class $\Fm(\Omega)$.
\item[(ii)] When $\chi(t) =-(-t)^p$, then $\Ecm(\Omega)$ is the class $\Epm(\Omega)$.
\item[(iii)] When $\chi(t)=-1-(-t)^p$, then $\Ecm(\Omega)$ is exactly the class $\Fpm(\Omega)$.
\end{itemize}
\end{remark}

First, we prove that Hessian operator is well-defined on class $\Ecm$ if  $\chi\not\equiv 0$.
\begin{theorem}\label{Main theorem 1}
	Let $\chi:\R^-\rightarrow\R^-$ be an increasing function. Then
	\begin{itemize}
		\item[(i)] $\Ecm(\Omega)\subset\Em(\Omega)$ if $\chi\not\equiv 0$,
		\item[(ii)] $\Ecm(\Omega)\subset\mathcal{N}_m(\Omega)$ if $\chi(-t)<0$ for all $t>0$.
	\end{itemize}   
\end{theorem}
The proof of Theorem \ref{Main theorem 1} requires the following auxiliary lemma.
\begin{lemma}\label{Lemma 1}
	Let $\chi:\R^-\rightarrow\R^-$ be an increasing function and $u,v\in\Eom(\Omega)$ satisfying $$(-\chi\circ u)(dd^cu)^m\wedge\omega^{n-m}=(dd^cv)^m\wedge\omega^{n-m}.$$ Then
	$$u\geq\frac{1}{\sqrt[m]{-\chi(-t)}}v-t,\ \forall t>0\text{ such that }\chi(-t)<0.$$
\end{lemma}
\begin{proof}
	Let $t>0$ such that $\chi(-t)<0$, we set $w:=\frac{1}{\sqrt[m]{-\chi(-t)}}v-t$.
	On the set $\{u\geq-t\}$, we have $u\geq-t\geq w$ since $v\leq 0$. It remains to prove that $u\geq w$ on the set $\{u<-t\}$. To do this, we observe that
	$$\int\limits_{\{u<-t\}}(dd^cu)^m\wedge\omega^{n-m}\leq \int\limits_{\{u<-t\}}\frac{-\chi\circ u}{-\chi(-t)}(dd^cu)^m\wedge\omega^{n-m}=\int\limits_{\{u<-t\}}(dd^cw)^m\wedge\omega^{n-m}.$$
	Then, by \cite[Theorem 2.13]{Lu15} and the fact that $$\lim\limits_{\{u<-t\}\ni z\rightarrow\partial\{u<-t\}}u(z)=-t\geq \lim\limits_{\{u<-t\}\ni z\rightarrow\partial\{u<-t\}}w(z),$$ 
	we have $u\geq w$, as desired.
\end{proof}
\begin{proof}[Proof of Theorem \ref{Main theorem 1}]
	Let $u\in\Ecm(\Omega)$ and $(u_j)$ be a sequence in $\Eom(\Omega)$ decreasing to $u$ such that
	\begin{equation}\label{sup finite}
	\sup\limits_j\int\limits_\Omega(-\chi\circ u_j)(dd^cu_j)^m\wedge\omega^{n-m}<\infty.
	\end{equation}
	By inequality (\ref{sup finite}), for every $j$, we can choose a positive number $M_j$ large enough such that $(dd^c(M_j u_j))^m\wedge\omega^{n-m}\geq (-\chi\circ u_j)(dd^cu_j)^m\wedge\omega^{n-m}$. By the fact that $M_j u_j\in\Eom(\Omega)$ and by \cite[Theorem 5.9]{VN17}, it follows that there exists $v_j\in\Em(\Omega)$ such that 
	\begin{equation}\label{equality}
	(dd^cv_j)^m\wedge\omega^{n-m}=(-\chi\circ u_j)(dd^cu_j)^m\wedge\omega^{n-m}
	\end{equation} and $v_j\geq M_ju_i$. We also have $v_j\in\Eom(\Omega)$ since $M_j u_j\in\Eom(\Omega)$. Then, by Lemma \ref{Lemma 1}, we have
	\begin{equation}\label{inequality uj vj}
	u_j\geq \frac{1}{\sqrt[m]{-\chi(-t)}}v_j-t,\ \forall t>0\text{ such that }\chi(-t)<0.
	\end{equation}
	
	Now we set $w_j=(\sup\limits_{k\geq j}v_k)$ and $w=\lim\limits_{j\rightarrow\infty}w_j$. By (\ref{inequality uj vj}), we have
	\begin{equation}\label{inequality u w}
	u\geq \frac{1}{\sqrt[m]{-\chi(-t)}}w^*-t,\ \forall t>0\text{ such that }\chi(-t)<0.
	\end{equation}
	By inequalities \ref{sup finite} and \ref{equality}, it follows that
	$$\sup_{j}\int\limits_\Omega(dd^c w_j^*)^m\wedge\omega^{n-m}\leq\sup\limits_j \int\limits_\Omega(dd^c v_j)^m\wedge\omega^{n-m}=	\sup\limits_j\int\limits_\Omega(-\chi\circ u_j)(dd^cu_j)^m\wedge\omega^{n-m}<\infty.$$
	Then, by the fact that $w_j^*\in\Eom(\Omega)$ and $w_j^*\searrow w^*$, it follows that $w^*\in\Fm(\Omega)$.
	
	(i) Assume that $\chi\not\equiv0$. Then there exists $t_0>0$ such that $\chi(t_0)<0$. Therefore,
	$$u\geq \frac{1}{\sqrt[m]{-\chi(-t_0)}}w^*-t_0,$$
	and hence $u\in\Em(\Omega)$, as desired.
	
	(ii) Assume that $\chi(-t)<0$ for all $t>0$. Then, by the fact that $\tilde{w^*}=0$ and by (\ref{inequality u w}), we have $\tilde{u}\geq -t$ for all $t>0$. Therefore, $\tilde{u}=0$ and hence $u\in\mathcal{N}_m(\Omega)$, as desired.
\end{proof}
\begin{remark}
		The condition of $\chi$ in Theorem \ref{Main theorem 1} part (ii) is sharp. Indeed, suppose that there is $t_0>0$ such that $\chi(-t_0)=0$. 
		We define $u\equiv-t_0$. Then we can find a sequence $(u_j)\in\Eom(\Omega)$ decreasing to $u$.
		By the increasing of $\chi$, we have $\chi(u_j)=0$ for every $j$, and hence
		$$\int\limits_\Omega(-\chi\circ u_j)(dd^cu_j)^m\wedge\omega^{n-m}=0,$$
		for every $j$. It follows that $u\in\Ecm(\Omega)$. But $u\not\in\mathcal{N}_m(\Omega)$ since the boundary value of $u$ is a nonzero (see Remark \ref{re:cegrell} (ii)).
\end{remark}
The following theorem shows that the $m-$weighted energies of functions of classes $\Ecm$ are finite.
\begin{theorem}\label{finite m-weighted energy}
	Let $\chi:\R^-\rightarrow\R^-$ be an increasing function  and $u\in\Ecm(\Omega)$. Then $$\int\limits_\Omega(-\chi\circ u)(dd^cu)^m\wedge\omega^{n-m}<\infty.$$
\end{theorem}
\begin{proof}
	Let $(u_j)$ be a sequence in $\Eom(\Omega)$ decreasing to $u$ such that
	$$M:=\sup\limits_j\int\limits_\Omega(-\chi\circ u_j)((dd^cu_j)^m\wedge\omega^{n-m})<\infty.$$
	Also, we have $(dd^cu_j)^m\wedge\omega^{n-m}\rightharpoonup (dd^cu)^m\wedge\omega^{n-m}$. 
	
	First, we consider the case where $\chi$ is a continuous function. Then, since $(-\chi\circ u_j)\nearrow(-\chi\circ u)$ and all of them are lower continuous, we have
	\begin{align*}
	\int\limits_\Omega(-\chi\circ u)(dd^cu)^m\wedge\omega^{n-m}&=\lim\limits_{j\rightarrow\infty}	\int\limits_\Omega(-\chi\circ u)(dd^cu_j)^m\wedge\omega^{n-m}\\
	&=\lim\limits_{j\rightarrow\infty}	\int\limits_\Omega\liminf\limits_{k\rightarrow\infty}(-\chi\circ u_k)(dd^cu_j)^m\wedge\omega^{n-m}\\
	&\leq\lim\limits_{j\rightarrow\infty}\liminf\limits_{k\rightarrow\infty}	\int\limits_\Omega(-\chi\circ u_k)(dd^cu_j)^m\wedge\omega^{n-m}\\
	&\leq \liminf\limits_{j\rightarrow\infty}	\int\limits_\Omega(-\chi\circ u_j)(dd^cu_j)^m\wedge\omega^{n-m}\\&=M<\infty.
	\end{align*}
	
	Next, we consider the case $\chi(-\infty)>-\infty$. We can find a sequence of increasing continuous function $(\chi_p)$ such that $\chi_p\searrow\chi$ on $\R^-$. The previous case implies that
	$$\int\limits_\Omega(-\chi_p\circ u)(dd^cu)^m\wedge\omega^{n-m}<\infty,$$
	for every $p$. Let $p\rightarrow\infty$, by Lebesgue monotone convergence theorem, we have
	$$\int\limits_\Omega(-\chi\circ u)(dd^cu)^m\wedge\omega^{n-m}<\infty.$$
	
	In general case, for each $q\in\N$, we set $\chi_q=\max (\chi,-q)$. By the previous case, we have
	$$\int\limits_\Omega(-\chi_q\circ u)(dd^cu)^m\wedge\omega^{n-m}<\infty,$$
	for every $q$. Again, letting $q\rightarrow\infty$ and using Lebesgue monotone convergence theorem, we get 
	$$\int\limits_\Omega(-\chi\circ u)(dd^cu)^m\wedge\omega^{n-m}<\infty.$$
	The proof is completed.
\end{proof}
Next, we consider the condition $\chi(-t)$ decreases to $-\infty$ when $t$ increases to $+\infty$.
The following theorem shows that the Hessian measure of functions of classes $\Ecm$ vanish on $m-$polar sets.
\begin{theorem}\label{chi=-infty}
	Let $\chi:\R^-\rightarrow\R^-$ be an increasing function. Then $\Ecm(\Omega)\subset\Eam(\Omega)$ if and only if $\chi(-\infty)=-\infty$.
\end{theorem}
\begin{proof}
	First, we suppose that $\chi(-\infty)=-\infty$. 
	Let $u\in\Ecm(\Omega)$ and $(u_j)$ be a sequence in $\Eom(\Omega)$ decreasing to $u$ such that
	$$M:=\sup\limits_j\int\limits_\Omega(-\chi\circ u_j)(dd^cu_j)^m\wedge\omega^{n-m}<\infty.$$
	By Theorem \ref{Main theorem 1}, we have $u\in\Em(\Omega)$. It remains to show that the measure $(dd^cu)^m\wedge\omega^{n-m}$ vanishes on every $m-$polar set. However, this measure vanishes on every $m-$polar set $E\subset\{u>-\infty\}$ thanks to Theorem \ref{vanish on polar sets}. Thus, we only need to show that $(dd^cu)^m\wedge\omega^{n-m}$ vanishes on $\{u=-\infty\}$.
	Indeed, by the increasing of $\chi$, we have, for every $j,k,$ and for every $t$ such that $\chi(-t)\neq0$,
	$$\int\limits_{\{u_j<-t\}}(dd^cu_k)^m\wedge\omega^{n-m}\leq \frac{1}{-\chi(-t)}\int\limits_{\{u_j<-t\}}(-\chi\circ u_k)(dd^cu_k)^m\wedge\omega^{n-m}\leq\frac{M}{-\chi(-t)}.$$ 
	Let $j\rightarrow\infty$, we obtain, for every $k$, and for every $t$ such that $\chi(-t)\neq0$,
	$$\int\limits_{\{u<-t\}}(dd^cu_k)^m\wedge\omega^{n-m}\leq\frac{M}{-\chi(-t)}.$$
	Next, let $k\rightarrow\infty$, we get, for every $t$ such that $\chi(-t)\neq0$,
	$$\int\limits_{\{u<-t\}}(dd^cu)^m\wedge\omega^{n-m}\leq\frac{M}{-\chi(-t)}.$$
	Finally, let $t\rightarrow\infty$, we have
	$$\int\limits_{\{u=-\infty\}}(dd^cu)^m\wedge\omega^{n-m}=0,$$ as desired.
	
	Now, we suppose that $\chi(-\infty)\neq-\infty$. Then, by the increasing of $\chi$, it follows that $\Fm(\Omega)\subset\Ecm(\Omega)$. However, we have $\Fm(\Omega)$ is a proper subset of $\Eam(\Omega)$ ( see Lemma \ref{lem:convexcone} (ii)). Therefore, $\Ecm(\Omega)$ is not a subset of $\Eam(\Omega)$.
	
	The proof is completed.
\end{proof}
If we further assume that $\chi(0)\neq0$, we have the following results about the relationship between classes $\Ecm$ and class $\Fm$.
\begin{theorem}\label{chi=0 main 1}
	Let $\chi:\R^-\rightarrow\R^-$ be an increasing function such that $\chi(0)\neq0$ and $\chi(-\infty)=-\infty$. Then
	$$\Ecm(\Omega)\subset\Fam(\Omega).$$
\end{theorem}
\begin{proof}
	Let $u\in\Ecm(\Omega)$. Then there exists a sequence $(u_j)\subset\Eom(\Omega)$ such that $u_j\searrow u$ and
	$$M:=\sup\limits_j\int\limits_\Omega(-\chi\circ u_j)(dd^cu_j)^m\wedge\omega^{n-m}<\infty.$$
	Hence, since $\chi$ is an increasing function and $\chi(0)\neq 0$, we have
	$$\int\limits_\Omega(dd^cu_j)^m\wedge\omega^{n-m}\leq\frac{1}{-\chi(0)}\int\limits_\Omega(-\chi\circ u_j)(dd^cu_j)^m\wedge\omega^{n-m}\leq\frac{M}{-\chi(0)}.$$
	This show that $$\sup\limits_j\int\limits_\Omega (dd^cu_j)^m\wedge\omega^{n-m}<+\infty.$$
	Therefore, we have $u\in\Fm(\Omega)$. Repeat the first part of proof of Theorem \ref{chi=-infty}, we have $(dd^cu)^m\wedge\omega^{n-m}$ does not charge on $m-$polar sets, and hence $u\in\Fam(\Omega)$, as desired.
\end{proof}
\begin{remark}
	The condition $\chi(0)\neq0$ of Theorem \ref{chi=0 main 1} is sharp. Indeed, in case $m=n$ and $\chi(t)=t$, Example 3.11 in \cite{C98} provides a counterexample.
\end{remark}
\begin{proposition}
	$$\Fm(\Omega)\cap L^\infty(\Omega)=\bigcap_{\chi\in\mathcal{X}}\Ecm(\Omega),$$
	where $\mathcal{X}:=\{\chi:\R^-\rightarrow\R^-:\ \chi \text{ is increasing},\ \chi(0)\neq0\text{  and } \chi(-\infty)=-\infty\}.$
\end{proposition}
\begin{proof}
	Suppose that $u\in \Fm(\Omega)\cap L^\infty(\Omega)$. Then there exists a sequence $(u_j)\subset\Eom(\Omega)$ such that $u_j\searrow u$ and
	$$\sup\limits_j\int\limits_\Omega(dd^cu_j)^m\wedge\omega^{n-m}<\infty.$$
	Therefore, for any $\chi\in\mathcal{X}$, we have
	\begin{align*}
	\sup\limits_j\int\limits_\Omega(-\chi\circ u_j)(dd^cu_j)^m\wedge\omega^{n-m}&
	\leq \sup\limits_j\Big(\sup\limits_\Omega|\chi\circ u_j|\Big)\int\limits_\Omega(dd^cu_j)^m\wedge\omega^{n-m}\\
	&\leq \Big(\sup\limits_\Omega|\chi\circ u|\Big)\sup\limits_j\int\limits_\Omega(dd^cu_j)^m\wedge\omega^{n-m}\\&<\infty,
	\end{align*}
	and so that $u\in\bigcap\limits_{\chi\in\mathcal{X}}\Ecm(\Omega)$.
	
	Conversely, suppose that $u\not\in \Fm(\Omega)\cap L^\infty(\Omega)$, we need to show that $u\not\in \bigcap\limits_{\chi\in\mathcal{X}}\Ecm(\Omega)$. By Theorem \ref{chi=0 main 1}, we can assume that $u\in\Fm(\Omega)\setminus L^\infty(\Omega)$. Then the sublevel set $\{u<-s\}$ are non empty open subsets for all $s>0$. Hence, by Theorem \ref{Inequalities for Capacity}, we have, for every $s>0$,
	$$\int\limits_{\{u<-s\}}(dd^cu)^m\wedge\omega^{n-m}\leq \text{Cap}_{m,\Omega}(\{u<-s-1\})\leq \text{V}_{2n}(\{u<-s-1\})>0.$$
	Therefore, we can consider the function $\chi_0:\R^-\rightarrow\R^-$ such that $\chi_0(0)\neq0$ and
	$$\chi_0'(-t)=\frac{1}{\int\limits_{\{u<t\}}(dd^cu)^m\wedge\omega^{n-m}},\text{ for all }t>0.$$
	Obviously, $\chi_0$ is increasing. Since $u\in\Fm(\Omega)$, we have $\int\limits_\Omega(dd^cu)^m\wedge\omega^{n-m}<\infty$, and so that $\chi_0'(-t)\geq \frac{1}{\int\limits_\Omega(dd^cu)^m\wedge\omega^{n-m}}>0$ for all $t>0$. This implies that $\chi_0(-\infty)=-\infty$. Therefore, $\chi_0\in\mathcal{X}$. Now
	\begin{align*}
	\int\limits_\Omega(-\chi_0\circ u)(dd^cu)^m\wedge\omega^{n-m}&=\int_0^{+\infty}(dd^cu)^m\wedge\omega^{n-m}\Big(\{u<\chi_0^{-1}(t)\}\Big)dt\\
	&=\int_0^{+\infty}\chi_0'(-s)(dd^cu)^m\wedge\omega^{n-m}\Big(\{u<-s\}\Big)ds\\
	&=\int_0^{+\infty} ds=\infty.
	\end{align*}
	It follows that $u\not\in\E_{m,\chi_0}(\Omega)$, and hence $u\not\in\bigcap\limits_{\chi\in\mathcal{X}}\Ecm(\Omega)$, as desired.
\end{proof}
The following theorem helps us to understand classes $\Ecm$ through the capacity of sublevel sets.

\begin{theorem}\label{thm:in-ca}
	Let $\chi:\R^-\rightarrow\R^-$ be an increasing function such that $\chi\in C^1(\R^-)$. Then
	$$\Ecm(\Omega)\supset\Bigg\{u\in\SHm^-(\Omega):\ \int\limits_0^\infty t^m\chi'(-t)\text{Cap}_{m,\Omega}(\{u<-t\})dt<+\infty  \Bigg\}.$$
\end{theorem}
\begin{proof}
	Suppose that $u\in\SHm^-(\Omega)$ satisfying $\int\limits_0^\infty t^m\chi'(-t)\text{Cap}_{m,\Omega}(\{u<-t\})dt<+\infty.$ By \cite[Theorem 3.1]{Lu15}. There exists sequence $\{u_j\}\subset\Eom(\Omega)$ such that $u_j\searrow u$ in $\Omega$. Then, by Theorem \ref{Inequalities for Capacity}, we have
	\begin{align*}
	\int\limits_\Omega(-\chi\circ u_j)(dd^cu_j)^m\wedge\omega^{n-m}&=\int\limits_0^{+\infty}(dd^cu_j)^m\wedge\omega^{n-m}\Big(\Big\{-\chi\circ u_j>t\Big\}\Big)dt\\
	&=\int\limits_0^{+\infty}\chi'(-s)(dd^cu_j)^m\wedge\omega^{n-m}\Big(\Big\{u_j<-s\Big\}\Big)ds\\
	&\leq \int\limits_0^{+\infty}\chi'(-s)s^m\text{Cap}_{m,\Omega}\Big(\Big\{u_j<-s\Big\}\Big)ds\\
	&\leq \int\limits_0^{+\infty}\chi'(-s)s^m\text{Cap}_{m,\Omega}\Big(\Big\{u<-s\Big\}\Big)ds.
	\end{align*}
	Therefore,
	$$\sup\limits_j \int\limits_\Omega(-\chi\circ u_j)(dd^cu_j)^m\wedge\omega^{n-m}<+\infty,$$
	and so that $u\in\Ecm(\Omega)$ as desired.
\end{proof}

In case $\chi(-t)<0$ for every $t>0$, the following lemma shows that a function of class $\mathcal{N}_m$ whose finite $m-$weighted energy can be approximated by a decreasing sequence of functions of class $\Eom$ whose $m-$weighted energies converge to the $m-$weighted energy of $u$.
\begin{lemma}\label{Approximate 2}
	Let $\chi:\R^-\rightarrow\R^-$ be an increasing function such that $\chi(-t)<0$ for every $t>0$, and $u\in\mathcal{N}_m(\Omega)$. Suppose that
	$$\int\limits_\Omega(-\chi\circ u)(dd^cu)^m\wedge\omega^{n-m}<+\infty.$$
	Then there exists a sequence $\{u_j\}\in\Eom(\Omega)$ such that $u_j\searrow u$ and 
	$$\lim\limits_{j\rightarrow\infty}\int\limits_\Omega(-\chi\circ u_j)(dd^cu_j)^m\wedge\omega^{n-m}=\int\limits_\Omega(-\chi\circ u)(dd^cu)^m\wedge\omega^{n-m}.$$
\end{lemma}
The proof use the same idea as in \cite{Ben11}.
\begin{proof}
	Let $\rho\in\Eom(\Omega)\cap C^\infty(\Omega)$ be a defining function for $\Omega$. For each $j$, by \cite[Lemma 5.5]{NVT19}, it follows that
	$$\mathds{1}_{\{u>j\rho\}}(dd^cu)^m\wedge\omega^{n-m}(\Omega)=\int\limits_{\{u>j\rho\}}(dd^cu_j)^m\wedge\omega^{n-m}\leq \int\limits_{\{u>j\rho\}}(dd^cj\rho)^m\wedge\omega^{n-m}<\infty.$$
	Then, by the fact that
	 $$(dd^cu)^m\wedge\omega^{n-m}\geq \mathds{1}_{\{u>j\rho\}}(dd^cu)^m\wedge\omega^{n-m},$$
	and by \cite[Theorem 5.9]{VN17}, it follows that there exists $u_j\in\Em(\Omega)$ such that $u_j\geq u$ and
	$$(dd^cu_j)^m\wedge\omega^{n-m}=\mathds{1}_{\{u>j\rho\}}(dd^cu)^m\wedge\omega^{n-m}.$$
	By Lemma \ref{vanish on polar sets}, we have  $(dd^cu_j)^m\wedge\omega^{n-m}$ vanishes on all $m-$polar subsets of $\Omega$ for every $j$.
	Moreover, by \cite[Lemma 5.1]{NVT19}, we have
	$$(dd^cu_j)^m\wedge\omega^{n-m}\leq \mathds{1}_{\{u>j\rho\}}(dd^c\max(u,j\rho)^m\wedge\omega^{n-m}\leq (dd^c\max(u,j\rho))^m\wedge\omega^{n-m}.$$
	Then, by \cite[Corollary 5.8]{NVT19}, it follows that $u_j\geq\max(u,j\rho)\geq j\rho$, and so that $u_j\in\Eom(\Omega)$.
	We observe that $(dd^cu_{j+1})^m\wedge\omega^{n-m}\geq (dd^cu_{j})^m\wedge\omega^{n-m}$ for every $j$. Again, by \cite[Corollary 5.8]{NVT19}, we have $u_j\geq u_{j+1}$ for every $j$. 	
	
	Next, let $t_0$ be a real number such that $\chi(-t_0)<0$. We choose an increasing function $\tilde{\chi}:\R^-\rightarrow\R^-$ such that  $\tilde{\chi}^{''}=\tilde{\chi}^{'}=0$ on $[-t_0,0]$, $\tilde{\chi}^{''}\geq0$ on $(-\infty, -t_0)$ and $\tilde{\chi}\geq\chi$ on $\R^-$. We have
	$$dd^c\tilde{\chi}(u_1)=\tilde{\chi}^{''}(u_1)du_1\wedge d^c u_1+\tilde{\chi}^{'}(u_1)dd^cu_1,$$
	and hence $(dd^c\tilde{\chi}(u_1))^m\geq0.$ Therefore, $\tilde{\chi}\circ u_1\in \SHm^-(\Omega)$. Also, we have $\tilde{\chi}\circ u_1\in L^\infty(\Omega)$, and so that $\tilde{\chi}\circ u_1\in \Em(\Omega).$ By the increasing of $\tilde{\chi}$, we have
	\begin{equation*}\label{tilde chi < infty}
	\int\limits_\Omega (-\tilde{\chi}\circ u_1)(dd^c u)^m\wedge\omega^{n-m}\leq 	\int\limits_\Omega (-\tilde{\chi}\circ u)(dd^c u)^m\wedge\omega^{n-m}<\infty.
	\end{equation*}
	
	Now we set $v:=\lim\limits_{j\rightarrow\infty}u_j$. We have $v\geq u$ and $$(dd^cv)^m\wedge\omega^{n-m}=\lim\limits_{j\rightarrow\infty}(dd^cu_{j})^m\wedge\omega^{n-m}=(dd^cu)^m\wedge\omega^{n-m}.$$ Then, applying Theorem \ref{u=v}, we have $u=v$, and so that $u_j\searrow u$.
	By monotone convergence, it follows that
	\begin{align*}
	\int\limits_\Omega(-\chi\circ u_j)(dd^cu_j)^m\wedge\omega^{n-m}&=\int\limits_\Omega(-\chi\circ u_j)\mathds{1}_{\{u>j\rho\}}(dd^cu)^m\wedge\omega^{n-m}\\
	&\rightarrow\int\limits_\Omega(-\chi\circ u)(dd^cu)^m\wedge\omega^{n-m}.
	\end{align*}
	The proof is completed.
\end{proof} 
By Theorem \ref{Main theorem 1}, Theorem \ref{finite m-weighted energy} and Lemma \ref{Approximate 2}, we have the following corollary.
\begin{corollary}\label{Approximate 1}
	Let $\chi:\R^-\rightarrow\R^-$ be an increasing function such that $\chi(-t)<0$ for every $t>0$  and $u\in\Ecm(\Omega)$. Then there exists a sequence $\{u_j\}\in\Eom(\Omega)$ such that $u_j\searrow u$ and 
	$$\lim\limits_{j\rightarrow\infty}\int\limits_\Omega(-\chi\circ u_j)(dd^cu_j)^m\wedge\omega^{n-m}=\int\limits_\Omega(-\chi\circ u)(dd^cu)^m\wedge\omega^{n-m}<+\infty.$$
\end{corollary}
By Lemma \ref{Approximate 2} and Theorem \ref{Main theorem 1}, we have the following relationships between class $\Ecm$ and Cegrell classes with finite $m-$weighted energies.
\begin{proposition}\label{pro:finiteint}
	Let $\chi:\R^-\rightarrow\R^-$ be an increasing function. Then
	\begin{itemize}
		\item[(i)] If $\chi\not\equiv0$ then 
		$$\Ecm(\Omega)\subset\Big\{u\in\Em(\Omega): \int\limits_\Omega(-\chi\circ u)(dd^cu)^m\wedge\omega^{n-m}<+\infty\Big\}.$$
		\item[(ii)] If $\chi(-t)<0$ for all $t>0$ then
		$$\Ecm(\Omega)=\Big\{u\in\mathcal{N}_m(\Omega): \int\limits_\Omega(-\chi\circ u)(dd^cu)^m\wedge\omega^{n-m}<+\infty\Big\}.$$
	\end{itemize}
\end{proposition}
\section{In a special case of $\chi$}\label{sec4}
In this section, we are going to consider a special condition for $\chi$: $\chi(-2t)\geq a\chi(-t)$ with some $a>1$. First, we recall a lemma which will be useful later.
\begin{lemma}\cite[Lemma 2]{VVH16}\label{Compa 1}
	Let $\chi:\R^-\rightarrow\R^-$ be an increasing function such that  $\chi(-2t)\geq a\chi(-t)$ with some $a>1$. Suppose that $u,v\in\Eom(\Omega)$. Then the following hold: 
	\begin{itemize}
		\item[(i)] If $u\leq v$ on $\Omega$, then
		$$\int\limits_\Omega(-\chi\circ v)(dd^cv)^m\wedge\omega^{n-m}\leq 2^m\max(a,2)\int\limits_\Omega(-\chi\circ u)(dd^cu)^m\wedge\omega^{n-m}.$$
		\item[(ii)] For every $0\leq \lambda\leq 1$, we have
		\begin{align*}
		&\int\limits_\Omega(-\chi\circ(\lambda u+(1-\lambda)v))(dd^c(\lambda u+(1-\lambda)v))^m\wedge\omega^{n-m}\\
		&\leq 2^m\max(a,2)\Big(\int\limits_\Omega(-\chi\circ u)(dd^cu)^m\wedge\omega^{n-m}+\int\limits_\Omega(-\chi\circ v)(dd^cv)^m\wedge\omega^{n-m}\Big).
		\end{align*}
	\end{itemize}
\end{lemma}
Our first goal is to relaxe the pointwise convergence
condition in the definition of $\Ecm$ to the almost everywhere convergence condition as follows.
\begin{theorem}\label{Convergence ae}
	Let $\chi:\R^-\rightarrow\R^-$ be an increasing function such that  $\chi(-2t)\geq a\chi(-t)$ with some $a>1$. Assume that there are
	$u_j\in\Eom (\Omega)$, $j\in\N$, such that $u_j$ converges almost everywhere to
	$u$ as $j\rightarrow\infty$ and
	$$\sup_{j>0}\int_{\Omega}(-\chi\circ u_j)(dd^cu_j)^m\wedge\omega^{n-m}<\infty.$$ Then $u\in\Ecm(\Omega)$.
\end{theorem}
\begin{proof}
		For every $k\geq 1$, we denote
		\begin{center}
			$u^k(z)=\sup\limits_{j\geq k}\max \{u, u_j\}$ and $v_k=(u^k)^*$.
		\end{center}
		Then, we have
		\begin{itemize}
			\item[(i)] $u^k$ converges to $u$ almost everywhere;
			\item [(ii)]$v_k:=(u^k)^*\in \SHm^-(\Omega)$ for all $k\geq 1$;
			\item [(iii)] $v_k$ is a decreasing sequence satisfying $v_k\geq u$ for every $k\geq 1$;
			\item [(iv)] $v_k=u^k$ almost everywhere.
		\end{itemize}
		By (i) and (iv),  we have $\lim\limits_{k\to\infty}v_k=u$ almost everywhere. Since 
		$u$ and $\lim\limits_{k\to\infty}v_k$ are subharmonic functions, we get $u=\lim\limits_{k\to\infty}v_k$.
		
		Since $ v_k\in\SHm^-(\Omega)$ and $v_k\geq u_k$, we have $v_k\in\Eom (\Omega)$. Then, by using Lemma
		\ref{Compa 1}, we obtain
		\begin{center}
			$C:=\sup\limits_{j>0}\int\limits_{\Omega}(-\chi\circ u_j)(dd^cu_j)^m\wedge\omega^{n-m}\geq\frac{1}{2^m\max(a,2)} \int\limits_{\Omega}(-\chi\circ v_k)(dd^cv_k)^m\wedge\omega^{n-m},$
		\end{center}
		for every $k\geq 1$.
		
		Now, it follows from \cite[Theorem 3.1]{Lu15} that there exists a decreasing sequence
		$w_k\in\Eom (\Omega)\cap C(\Omega)$ such that $\lim\limits_{j\to\infty}w_j(z)=u(z)$ in $\Omega$.
		Replacing $w_j$ by $(1-j^{-1})w_j$, we can assume that $w_j(z)>u(z)$
		for every $j>0, z\in\Omega$. Applying Lemma \ref{Compa 1}, we have, for every $j, k>0$,
		\begin{align*}
			\int\limits_{\{v_k<w_j\}}(-\chi\circ w_j)(dd^cw_j)^m\wedge\omega^{n-m}=& \int\limits_{\{v_k<w_j\}}(-\chi\circ \max(v_k,w_j))(dd^c\max(v_k,w_j))^m\wedge\omega^{n-m}\\
			\leq& \int\limits_{\Omega}(-\chi\circ \max(v_k,w_j))(dd^c\max(v_k,w_j))^m\wedge\omega^{n-m}\\
			\leq&\ 2^m\max(a,2) \int\limits_{\Omega}(-\chi\circ v_k))(dd^cv_k)^m\wedge\omega^{n-m}\\
			= &(2^m\max(a,2))^2\ C.
		\end{align*}
		Letting $k\rightarrow \infty$, we get,
		\begin{center}
			$\int\limits_{\Omega}(-\chi\circ w_j)(dd^cw_j)^n\leq (2^m\max(a,2))^2 C,$
		\end{center}
		for every $j>0$, and so that 
		$$\sup\limits_j\int\limits_{\Omega}(-\chi\circ w_j)(dd^cw_j)^m\wedge\omega^{n-m}<+\infty.$$
		Therefore, $u\in\Ecm(\Omega)$.
\end{proof}
As a consequence, we will see that the almost everywhere limit of a sequence of functions of class $\Ecm$ whose bounded $m-$weighted energies also belong to this class as follows.
\begin{theorem}\label{thm:convergence}
		Let $\chi:\R^-\rightarrow\R^-$ be an increasing function such that  $\chi(-2t)\geq a\chi(-t)$ with some $a>1$. Assume that there are
		$u_j\in\Ecm (\Omega)$, $j\in\N$, such that $u_j$ converges almost everywhere to
		$u$ as $j\rightarrow\infty$ and
		$$\sup_{j>0}\int_{\Omega}(-\chi\circ u_j)(dd^cu_j)^m\wedge\omega^{n-m}<\infty.$$ Then $u\in\Ecm(\Omega)$.
\end{theorem}
\begin{proof}
	By Corollary \ref{Approximate 1}, for each $u_j$, there exists a sequence $\{u_{j,k}\}_{k=1}^\infty\subset\Eom(\Omega)$ such that $u_{j,k}\searrow u_j$ when $k\rightarrow+\infty$ and
	$$\lim\limits_{k\rightarrow\infty} \int\limits_\Omega(-\chi\circ u_{j,k})(dd^cu_{j,k})^m\wedge\omega^{n-m}=\int\limits_\Omega(-\chi\circ u_{j})(dd^cu_{j})^m\wedge\omega^{n-m}.$$
	We set $v_j=u_{j,j}$. Then $\{v_j\}\subset\Eom(\Omega)$, $v_j$ converges almost everywhere to $u$ and
	$$\sup_{j>0}\int_{\Omega}(-\chi\circ v_j)(dd^cv_j)^m\wedge\omega^{n-m}<\infty.$$ 
	By Theorem \ref{Convergence ae}, we have $u\in\Ecm(\Omega)$ as desired.
\end{proof}
Our last purpose is to create elements in class $\Ecm$ by integrating.
\begin{theorem}\label{thm:integration}
		 Let $(X,d,\mu)$ be a totally bounded metric probability space on $\Omega$ and $\chi:\R^-\rightarrow\R^-$ be an increasing function such that  $\chi(-2t)\geq a\chi(-t)$ with some $a>1$. Assume that $u:\Omega\times X\rightarrow[-\infty,0)$ satisfies the following properties
		 \begin{itemize}
		 	\item[i,]  $u(\cdot,a)\in\Ecm(\Omega)$ for every $a\in X$,
		 	\item[ii,]  $\sup\limits_{a\in X}\ \int\limits_\Omega(-\chi\circ u(\cdot,a))(dd^cu(\cdot,a))^m\wedge\omega^{n-m}=M<+\infty,$
		 	\item[iii,] $u(z,\cdot)$ is upper semicontinuous on $X$ for every $z\in\Omega$.
		 \end{itemize}
		 Then
		 $$\tilde{u}(z)=\int\limits_X u(z,a)d\mu(a)\in\Ecm(\Omega).$$
\end{theorem}
\begin{proof}
	By \cite[Theorem 2.6.5]{Klimek}, it follows that the restriction of $\tilde{u}$ on $\Omega\cap W$ is either plurisubharmonic or $-\infty$ for every $(n-m+1)$-dimensional subspace $W$ of $\mathbb{C}^n$, so that $\tilde{u}\in\SHm(\Omega)$ or $\tilde{u}\equiv-\infty$.
	
	Since $X$ is totally bounded, we can divide $X$ into a finite pairwise disjoint collection of sets of diameter at most $\frac{1}{2}$. Since each set of this collection is also totally bounded, we can divide it into a finite pairwise disjoint collection of sets of diameter at most $\frac{1}{4}$. By repeating this procedure, for each $j\in\N$, we can divide $X$ into a finite pairwise disjoint collection $\{U^j_k\}_{k=1}^{m_j}$ such that $d(U^j_k)=\sup_{x,y\in U^j_k}\{\|x-y\|\}\leq\frac{1}{2^j}$ for every $1\leq k\leq m_j$, and there exists natural numbers $0=p_{j,0}< p_{j,1}<p_{j,2}<...<p_{j,m_{j-1}}< m_{j}$ such that, for $k\in\{1,2,...,m_{j}\}$,
	$$ U^j_k=\bigcup_{k'=p_{j,k-1}+1}^{p_{j,k}}U^{j+1}_{k'}.$$
	
	For $j\in\N$, we define 
	$$u_j(z)=\sum\limits_{k=1}^{m_j}\mu(U^j_k)\sup\limits_{a\in U^j_k}u(z,a)\text{ and }\tilde{u}_j=u_j^*.$$
	We first claim that $\tilde{u}_j\in\Eom(\Omega)$ for every $j$.
	We observe that $\tilde{u}_j\in\SHm^-(\Omega)$ for every $j\in\N$. Fix $j\in\N$, we choose arbitrarily $a_k\in U^j_k$ for every $k\in\{1,...,m_j\}$. Then $ \tilde{u}_j\geq\sum\limits_{k=1}^{m_j}\mu(U^j_k)u(\cdot,a_k)\in\Eom(\Omega)$, and hence $\tilde{u}_j\in\Eom(\Omega)$, as claimed. 
	
	Next, we claim that $\tilde{u}_j\searrow \tilde{u}$. To do this, we first show that $u_j \searrow \tilde{u}$. Indeed, we observe that
	\begin{align*}
	u_j(z)&=\sum\limits_{k=1}^{m_j}\mu(U^j_k)\sup\limits_{a\in U^j_k}u(z,a)=\sum\limits_{k=1}^{m_j}\Big(\sum\limits_{k'=p_{j,k-1}+1}^{p_{j,k}}\mu(U^{j+1}_{k'})\sup\limits_{a\in U^j_k}u(z,a)\Big)\\
	&\geq \sum\limits_{k=1}^{m_j}\Big(\sum\limits_{k'=p_{j,k-1}+1}^{p_{j,k}}\mu(U^{j+1}_{k'})\sup\limits_{a\in U^{j+1}_{k'}}u(z,a)\Big)\\
	&= \sum\limits_{k'=1}^{m_{j+1}}\mu(U^{j+1}_{k'})\sup\limits_{a\in U^{j+1}_{k'}}u(z,a)\\
	&=u_{j+1}(z),
	\end{align*}
	and 
	$$ u_j(z)=\int\limits_X\sum\limits_{k=1}^{m_j}\mathds{1}_{U^j_k}(a)\sup_{a\in U^j_k}u(z,a)d\mu(a)\geq \int\limits_X\sum\limits_{k=1}^{m_j}\mathds{1}_{U^j_k}(a)u(z,a)d\mu(a)=\tilde{u}(z),$$
	where
	 $\mathds{1}_{U^j_k}$ is the characteristic function of ${U^j_k}$. Also, by the upper semicontinuity of $u(z,\cdot)$, it follows that
	 $$u(z,a)\geq \lim\limits_{j\rightarrow\infty}\Big(\sup\limits_{|b-a|\leq 2^{-j}}u(z,b)\Big)\geq \lim\limits_{j\rightarrow\infty}\sum\limits_{k=1}^{m_j}\mathds{1}_{U^j_k}(a)\sup_{a\in U^j_k}u(z,a),$$
	 and then, by Fatou's lemma, we obtain
	 $$\tilde{u}(z)=\int\limits_Xu(z,a)d\mu(a)\geq \lim\limits_{j\rightarrow\infty}\sum\limits_{k=1}^{m_j}\int\limits_X\mathds{1}_{U^j_k}(a)\sup_{a\in U^j_k}u(z,a)d\mu(a)=\lim\limits_{j\rightarrow\infty} u_j(z).$$
	Therefore, $u_j\searrow\tilde{u}$. Since $u_j=\tilde{u}_j$ almost everywhere, then $\tilde{u}_j\searrow\tilde{u}$ almost everywhere. Then, by the fact that $\lim\limits_{j\rightarrow\infty}\tilde{u}_j$ is either plurisubharmonic or identically $-\infty$, we have $\tilde{u}_j\searrow\tilde{u}$ everywhere, as claimed.
	
	Finally, we claim that: $$\sup\limits_{j\in\N}\int\limits_\Omega(-\chi\circ \tilde{u}_j)(dd^c\tilde{u}_j)^m\wedge\omega^{n-m}< \infty.$$
	Indeed, by Lemma \ref{Compa 1} part (i), it follows that
	\begin{align*}
	&\int\limits_\Omega(-\chi\circ \tilde{u}_j)(dd^c\tilde{u}_j)^m\wedge\omega^{n-m}\\
	&\leq 2^m\max(a,2)\int\limits_\Omega\Big(-\chi\circ \big(\sum\limits_{k=1}^{m_j}\mu(U^j_k)u(z,a_k)\big)\Big)\Big(dd^c\big(\sum\limits_{k=1}^{m_j}\mu(U^j_k)u(z,a_k)\big)\Big)^m\wedge\omega^{n-m}.
	\end{align*}
	Then by Lemma \ref{Compa 1} part (ii), and the fact that $\sum\limits_{k=1}^{m_j}\mu(U^j_k)=1$, we obtain, for every $j\in\N$,
	\begin{align*}
	&\int\limits_\Omega(-\chi\circ \tilde{u}_j)(dd^c\tilde{u}_j)^m\wedge\omega^{n-m}\\
	&\leq2^{2m}\max(a^2,4)\sum\limits_{k=1}^{m_j}\int\limits_\Omega\Big(-\chi\circ u(z,a_k)\Big)\Big(dd^c\big(u(z,a_k)\big)\Big)^m\wedge\omega^{n-m}\\
	&\leq 2^{2m}\max(a^2,4)M,
	\end{align*}
	which proves our claim.
	
	We have proved that $\{\tilde{u}_j\}\subset\Eom(\Omega)$, $\tilde{u}_j\searrow \tilde{u}$ and $$\sup\limits_{j\in\N}\int\limits_\Omega(-\chi\circ \tilde{u}_j)(dd^c\tilde{u}_j)^m\wedge\omega^{n-m}< \infty.$$
	Then, by the definition of class $\Ecm$, it follows that $\tilde{u}\in\Ecm(\Omega)$ (and, as a consequence, $\tilde{u}\not\equiv-\infty$), as desired.
\end{proof}

\section*{Acknowledgments}
The part of this work was done while the authors were visiting to Vietnam Institute for Advanced
Study in Mathematics (VIASM). The authors would like to thank the VIASM for
hospitality and support. 

\end{document}